\documentclass{amsart}

\usepackage{amssymb}
\usepackage{amsmath}
\usepackage{amsfonts,mathrsfs}
 \usepackage{url}
\newtheorem{theorem}{Theorem}

\theoremstyle{plain}

\newtheorem{corollary}{Corollary}

\newtheorem{definition}{Definition}
\newtheorem{example}{Example}

\newtheorem{proposition}{Proposition}

\numberwithin{equation}{section}


\title{A Generalization of the Lebesgue Density Theorem via Modulus Density}
\author{}
\date{}

\begin{document}
\vspace*{-3cm}
\maketitle


\begin{center}
{\large H. S. Behmanush\textsuperscript{1,$\ast$} \qquad
M. K\"u\c{c}\"ukaslan\textsuperscript{1,\dag }\par}

\vspace{0.5em}

\textsuperscript{1}Mersin University, Dept.\ of Mathematics, Mersin, Turkey\par

\vspace{0.3em}

\textsuperscript{$\ast$} h.s.behmanush1989@gmail.com \par
\textsuperscript{\dag}mkkaslan@gmail.com, \textsuperscript{\dag}mkucukaslan@mersin.edu.tr
\end{center}

\vspace{1em}
\maketitle

\begin{abstract}
In this paper, we introduce the notion of a $\gamma$-density point for Lebesgue-measurable subsets of $\mathbb{R}$, where $\gamma$ is a modulus function, and study its basic measure-theoretic properties. We show that every $\gamma$-density point is a Lebesgue density point, while under Condition~(A) the two notions coincide. Consequently, for such modulus functions, the set of $\gamma$-density points of a measurable set differs from the set itself only by a null set, yielding a modulus version of the Lebesgue Density Theorem. We then define the associated $\gamma$-density topology $\tau_\gamma$ and investigate its structure. In general, $\tau_\gamma$ is contained in the classical Lebesgue density topology, and if $\gamma$ satisfies Condition~(A), then $\tau_\gamma=\tau_d$. We also compare $\tau_\gamma$ with $\psi$-density topologies and establish several topological properties of $\tau_\gamma$, including that countable sets are $\tau_\gamma$-closed and that $(\mathbb{R},\tau_\gamma)$ is nonseparable, nonregular, and nonmetrizable. Finally, we introduce $\gamma$-approximately continuous functions, prove that they form a vector space, and show that the bounded class of such functions is a Banach space under the supremum norm. 
\end{abstract}

\maketitle

\section{Introduction}
\label{Sec:1}
In this paper, we introduce a modulus-based extension of the classical notion of density for Lebesgue measurable subsets of $\mathbb{R}$. More precisely, for a measurable set $A\subset \mathbb{R}$ and a modulus function $\gamma$, we define the notion of a modulus density point of $A$ and investigate its measure-theoretic, topological, and functional consequences. This leads naturally to the associated modulus density topology, viewed as a generalization of the classical Lebesgue density topology, and to a corresponding notion of modulus-approximate continuity.

The concepts of density points and approximately continuous functions are fundamental in real analysis and measure theory, and were developed extensively during the twentieth century. Haupt and Pauc first introduced the density topology in 1952 and 1954 \cite{ hp1, hp2}. It was later rediscovered by Goffman and Waterman in 1961 \cite{cgcntnd}.  

A Lebesgue measurable set $A\subset \mathbb{R}$ is said to have density $d$ at a point $x$ if the following limit exists and is equal to $d$:
$$\lim_{\alpha \to 0}\frac{|A\cap (x-\alpha,x+\alpha)|}{|2\alpha|}.$$
The point $x$ is called a Lebesgue density point of $A$ if and only if $d=1$. 

In \cite{jcom}, J. C. Oxtoby presented the classical Lebesgue Density Theorem, which asserts that for every measurable set $A$, the symmetric difference between $A$ and the set of its density points is a null set.

The notation $\Phi(A)$ stands for the set of all points in $\mathbb{R}$ at which $A$ has density $1$. Let $\mathcal{L}$ denote the family of all Lebesgue measurable subsets of $\mathbb{R}$. Using the operator $\Phi: \mathcal{L} \to \mathcal{L}$, one defines the topology
$$\tau_d=\{A\subset \mathbb{R}: A\subset \Phi(A)\}.$$
It is well known that the Lebesgue density topology is strictly finer than the usual Euclidean topology on $\mathbb{R}$, and therefore provides a more delicate topological structure \cite{wh}.

In recent years, several authors have constructed new topologies by means of analogous operator techniques. In 2013, J. Hejduk and R. Wiertelak introduced abstract density topologies generated by lower and almost lower density operators and studied their principal properties \cite{hwro}. Later, in 2025, J. Hejduk and P. Nowakowski introduced the strong generalized topology generated by the porosity operator \cite{hp}. Earlier, in 2023, J. Hejduk, M. K\"u\c c\"ukaslan and A. Loranty defined the $\langle s\rangle$-generalized topology \cite{hma}. More recently, in 2024, J. Hejduk and A. Loranty investigated a class of generalized topologies satisfying all separation axioms \cite{ha}.

Among density-type operators, $\psi$-density plays an important role. In 1999, M. Terepeta and E. Wagner-Bojakowska introduced the concept of the $\psi$-density topology in analogy with the classical density topology \cite{mepsi}. This idea originates in observations from Taylor's notes. In \cite{sjto}, S. J. Taylor examined whether the Lebesgue Density Theorem can be strengthened. More precisely, if $A$ is a Lebesgue measurable set, then
$$\lim _{x\in I,|I|\to 0}\frac{|A^c \cap I|}{|I|}=0$$
holds for all $m\in A$, except possibly on a set of Lebesgue measure zero. Under the stronger condition
$$\lim _{x\in I,|I|\to 0}\frac{|A^c \cap I|}{|I|\psi(|I|)}=0,$$
where $\psi  \in \mathcal{C}$, and $\mathcal{C}$ denotes the class of all functions defined on the positive reals that are continuous, non-decreasing, and tend to zero at the origin, Taylor showed that the classical theorem is, in a certain sense, optimal: such an improvement cannot hold uniformly for every measurable set, although for each measurable set there exists a suitable function $\psi$ for which the stronger relation remains valid almost everywhere on $A$. 

In \cite{mepsi}, M. Terepeta and E. Wagner-Bojakowska restricted attention to intervals centered at $x$ and defined a point $x$ to be a $\psi$-density point of a set $A$ if and only if
$$\lim_{\alpha \to 0}\frac{|A^c \cap (x-\alpha,x+\alpha)|}{|2\alpha|.\psi(2\alpha)}=0.$$
They denoted the set of all $\psi$-density points of a measurable set $A$ by $\Phi_{\psi}(A)$. The family
$$\tau_{\psi}:=\{A\subset \mathbb{R}: A\subset \Phi_{\psi}(A)\}$$
is called the $\psi$-density topology, and it is finer than the Lebesgue density topology.

Two $\psi$-density topologies generated by functions $\psi_1$ and $\psi_2$ were shown to be equivalent under necessary and sufficient conditions by E. Wagner-Bojakowska and W. Wilczynski in \cite{wewc}. Since then, both $\psi$-density points and $\psi$-density topologies have been studied extensively. In 2000, results concerning the interior of a set in the $\psi$-density topology were obtained in \cite{wewt}. In 2006, M. Filipczak and M. Terepeta investigated the natural, density, and $\psi$-density topologies on $\mathbb{R}$ in connection with classes of continuous functions, comparing them with Baire 1, Baire $^*$1, and Darboux functions \cite{fmmto}. In 2009, M. Terepet studied several aspects of $\psi$-continuity in \cite{tma}, while A. Gozdziewicz-Smejda and E. Lazarow examined comparisons of $\psi$-sparse topologies in \cite{gawc}. Later, E. Lazarowa and A. Vizvaryb introduced the notions of $\psi_{I}$-density point and $\psi_{I}$-density topology, defined analogously to the classical $I$-density topology on the real line \cite{teavpsi}.

These developments motivate the following program. First, can a modulus function $\gamma$ generate a natural notion of density point for Lebesgue measurable subsets of $\mathbb{R}$ that extends the classical concept? Second, under what assumptions does this generalized notion yield an analogue of the Lebesgue Density Theorem, and how does the induced topology compare with the classical density topology and with $\psi$-density topologies? Third, can this framework be used to define a corresponding notion of modulus-approximate continuity, and what algebraic and functional-analytic properties do the resulting classes of functions possess?

The aim of this paper is to answer these questions. In Section 1, we introduce modulus density points and establish their basic measure-theoretic properties. In particular, we show that every modulus density point is a Lebesgue density point, and that under a natural additional assumption on the modulus function the two notions coincide, which leads to a modulus version of the Lebesgue Density Theorem. In Section 2, we define the corresponding modulus density topology and investigate its fundamental properties, including its relation to the Lebesgue density topology and to $\psi$-density topologies. In Section 3, we introduce modulus-approximately continuous functions and study the linear and functional-analytic structure of the associated function spaces.

Throughout this article, the symbol $|.|$ denotes the Lebesgue measure of a set. We write $X \sim Y$ to mean that the sets $X$ and $Y$ differ only by a null set; that is, $|X \Delta Y| = 0$, where $\Delta$ denotes the symmetric difference.

\begin{definition}\cite{AP}
    A function $\gamma : \mathbb{R}^+:=[0,\infty)\to [0,\infty)$ is said to be a modulus function if
$(i)$ $\gamma(a)=0 \Leftrightarrow a=0,$
    
$(ii)$ $\gamma(a+b)\leq \gamma(a)+\gamma (b)$ for all $ a,b\in \mathbb{R}^+,$
        
$(iii)$ $\gamma$ is right-continuous at $0$,  
       
$(iv)$ $\gamma$ is increasing.
\end{definition}
 
From the property $(ii)$, it follows that if $a\geq b,$ then $\gamma(a)-\gamma (b)\leq \gamma(a-b)$ holds for any modulus function in the set  $\mathcal{M}$.
The set of all modulus functions in this work will be  denoted by $\mathcal{M}$.
The functions $\gamma(x)=\frac{x}{1+x}$  and $\gamma(x)=x^p$ where $0 < p < 1$, etc. are examples of modulus functions.
\begin{definition}
    Let $A\in \mathcal{L}$ and $\gamma \in \mathcal{M}$, the right (or left) $\gamma$-density of $A$ at a point $x$ is defined as $$d_{\gamma}^+(A,x)=\lim_{\alpha \to 0^+}\frac{\gamma(|(x,x+\alpha)\cap A|)}{\gamma(\alpha)},$$ $$ d_{\gamma}^-(A,x)=\lim_{\alpha \to 0^+}\frac{\gamma(|(x-\alpha,x)\cap A|)}{\gamma(\alpha)}).$$
    Similarly,  $\gamma$-density of $A$ at a point $x$ is defined as $$d_{\gamma}(A,x)=\lim_{\alpha \to 0^+}\frac{\gamma(|(x-\alpha,x+\alpha)\cap A|)}{\gamma(2\alpha)}.$$
   
    A point $x$ is said to be $\gamma$-density point(or modulus density point) of $A$  if $d_{\gamma}(A^c,x)=0$ and 
     said to be a right (or left) $\gamma$-density point of set $A$ or a right (or left) modulus density point of $A$ if $d_{\gamma}^+(A^c,x)=0$ $ 
     ( d_{\gamma}^-(A^c,x)=0).$
     \end{definition}
     
   For $A\in \mathcal{L}$ and $\gamma \in \mathcal{M}$, the symbols $\mathcal{D}_{\gamma}(A)$,  $\mathcal{D}_{\gamma}^+(A)$($\mathcal{D}_{\gamma}^-(A)$) represent  the set of modulus density points and the set of right (left) $\gamma$-density points of $A$, respectively. 
   Taking $\gamma(x)=x$, we have $\mathcal{D}_{\gamma}(A)=\Phi_d(A)$
    where $\Phi_d(A)$ denotes the set of Lebesgue density points of set A.
\begin{definition}
A point $x$ is called $\gamma$-dispersion (modulus-dispersion) point of set $A$ if $x$ is a modulus-density point of $A^c$. That is
    $$\lim_{\alpha \to 0}\frac{\gamma(|(x-\alpha,x+\alpha)\cap A|)}{\gamma(2\alpha)}=0.$$
\end{definition}

\begin{theorem}
A point $x$ is a modulus density point of a set $A\in \mathcal{L}$ for any $\gamma \in \mathcal{M}$, if and only if it is a left and right modulus density point of the set $A$.
 \end{theorem}
 
\begin{proof} 
Assume \(x\in \mathcal{D}_\gamma(A)\). Then
$$
\lim_{\alpha\to 0^+}
\frac{\gamma(|(x-\alpha,x+\alpha)\cap A^c|)}{\gamma(2\alpha)}
=0.
$$
Since
$$
|(x-\alpha,x)\cap A^c|
\leq
|(x-\alpha,x+\alpha)\cap A^c|,
$$
monotonicity gives
$$
\gamma(|(x-\alpha,x)\cap A^c|)
\leq
\gamma(|(x-\alpha,x+\alpha)\cap A^c|).
$$
Also, by subadditivity,
$$
\gamma(2\alpha)=\gamma(\alpha+\alpha)\leq 2\gamma(\alpha),
$$
hence
$$
\frac{1}{\gamma(\alpha)}\leq \frac{2}{\gamma(2\alpha)}.
$$
Therefore
$$
0\leq
\frac{\gamma(|(x-\alpha,x)\cap A^c|)}{\gamma(\alpha)}
\leq
2\frac{\gamma(|(x-\alpha,x+\alpha)\cap A^c|)}{\gamma(2\alpha)}
\to 0.
$$
The same argument works for the right side, so \(x\in \mathcal{D}_\gamma^{-}(A)\cap \mathcal{D}_\gamma^{+}(A)\).
To prove reverse, 
Since

$|(x-\alpha,x)\cap A^c|$ and $|(x,x+\alpha)\cap A^c|$ are disjoint
$$|(x-\alpha,x+\alpha)\cap A^c|=|(x-\alpha,x)\cap A^c|+|(x,x+\alpha)\cap A^c|$$

Now assume \(x\) is both a left and a right \(\gamma\)-density point of \(A\). By definition this means
$$\lim_{\alpha\to 0^+}\frac{\gamma(|(x-\alpha,x)\cap A^c|)}{\gamma(\alpha)}=0
\qquad \text{and} \qquad
\lim_{\alpha\to 0^+}\frac{\gamma(|(x,x+\alpha)\cap A^c|)}{\gamma(\alpha)}=0.$$

Using subadditivity of \(\gamma\),
$$\gamma(|(x-\alpha,x)\cap A^c|+|(x,x+\alpha)\cap A^c|)\le \gamma(|(x-\alpha,x)\cap A^c|)+\gamma(|(x,x+\alpha)\cap A^c|).$$

Hence
$$\frac{\gamma(|(x-\alpha,x+\alpha)\cap A^c|)}{\gamma(2\alpha)}
= \frac{\gamma(|(x-\alpha,x)\cap A^c|+|(x,x+\alpha)\cap A^c|)}{\gamma(2\alpha)}
$$ $$\le \frac{\gamma(|(x-\alpha,x)\cap A^c|)+\gamma(|(x,x+\alpha)\cap A^c|)}{\gamma(2\alpha)}.$$
Since \(\gamma\) is increasing and \(\alpha \leq 2\alpha\), we have

$$\gamma(\alpha)\leq \gamma(2\alpha),$$
so
$$\frac{\gamma(|(x-\alpha,x)\cap A^c|)+\gamma(|(x,x+\alpha)\cap A^c|)}{\gamma(2\alpha)}
\leq \frac{\gamma(|(x-\alpha,x)\cap A^c|)}{\gamma(\alpha)}+\frac{\gamma(|(x,x+\alpha)\cap A^c|)}{\gamma(\alpha)}.$$
Therefore
$$ 0\leq
\frac{\gamma(|(x-\alpha,x+\alpha)\cap A^c|)}{\gamma(2\alpha)} \leq \frac{\gamma(|(x-\alpha,x)\cap A^c|)}{\gamma(\alpha)}+\frac{\gamma(|(x,x+\alpha)\cap A^c|)}{\gamma(\alpha)}
\longrightarrow 0. $$
So $$\lim_{\alpha\to 0^+} \frac{\gamma(|(x-\alpha,x+\alpha)\cap A^c|)}{\gamma(2\alpha)}=0,$$
which means \(x\in \mathcal{D}_\gamma(A)\). Thus the converse holds.
\end{proof}

\begin{proposition}\label{pn1.5}
 If there exist $a,b>0$ and $\delta>0$ such that
$$ a \leq \frac{\gamma_1(t)}{\gamma_2(t)} \leq b $$
for $0<t<\delta$, then $$ D_{\gamma_1}(A)=D_{\gamma_2}(A) $$
for every $A\in \mathcal{L}$.
\end{proposition}
\begin{proof}
For any measurable set $A,$  $0\leq |(x-\alpha,x+\alpha)\cap A^c| \leq 2\alpha$, so for small $\alpha$,
$$ a\gamma_2(|(x-\alpha,x+\alpha)\cap A^c|)\leq \gamma_1(|(x-\alpha,x+\alpha)\cap A^c|)\leq b\gamma_2(|(x-\alpha,x+\alpha)\cap A^c|), $$
and
$$ a\gamma_2(2\alpha)\leq \gamma_1(2\alpha)\leq b\gamma_2(2\alpha). $$
Hence
$$ \frac{a}{b}\frac{\gamma_2(|(x-\alpha,x+\alpha)\cap A^c|)}{\gamma_2(2\alpha)} \leq
\frac{\gamma_1(|(x-\alpha,x+\alpha)\cap A^c|)}{\gamma_1(2\alpha)} \leq
\frac{b}{a}\frac{\gamma_2(|(x-\alpha,x+\alpha)\cap A^c|)}{\gamma_2(2\alpha)}. $$
So one ratio tends to $0$ iff the other does.
\end{proof}
\begin{theorem}\label{thm6}
For any $\gamma\in \mathcal{M}$,
every modulus density point of a measurable set $A$ 
is also a Lebesgue density point of 
$A$.
\end{theorem}
\begin{proof}
Let $x$ be a modulus density point of $A$, so we have $\lim_{\alpha \to 0} \frac{\gamma(|(x-\alpha,x+\alpha)\cap A^c|)}{\gamma(2\alpha)}=0.$
To show Lebesgue density, it is enough to prove
$$ \frac{|(x-\alpha,x+\alpha)\cap A^c|}{2\alpha} \to 0.$$
Suppose not. Then there exist $\varepsilon>0$ and $\alpha_n \to 0^+$ such that
$$|(x-\alpha_n,x+\alpha_n)\cap A^c| \geq \varepsilon(2\alpha_n).$$
Choose $N\in \mathbb{N}$ with $N_\varepsilon \geq 1$. Then
$$ 2\alpha_n \leq N |(x-\alpha_n,x+\alpha_n)\cap A^c|. $$
By monotonicity and subadditivity,
$$ \gamma(2\alpha_n)\leq \gamma(N|(x-\alpha_n,x+\alpha_n)\cap A^c|)\leq N\gamma(|(x-\alpha_n,x+\alpha_n)\cap A^c|), $$
\end{proof}
It should be noted that the converse of the theorem is not valid in general. An illustrative example is provided to demonstrate this.
\begin{example} \label{ex1.7}
The converse of Theorem \ref{thm6} is false.
Define a function $\gamma:[0,\infty)\to[0,\infty)$ by
$$\gamma(0)=0,$$
and for $t>0$,
$$
\gamma(t)=
\begin{cases}
\dfrac{1}{\log \!\left(\dfrac{e}{t}\right)}, & 0<t\le e^{-1},\\[2mm]
\dfrac{e}{4}\,t+\dfrac14, & t>e^{-1}.
\end{cases}
$$
Then $\gamma$ is a modulus function. Indeed, it is increasing, right-continuous at $0$, and
$\gamma(0)=0$. Moreover, $\gamma$ is concave on $(0,\infty)$, and every increasing concave
function $\gamma$ with $\gamma(0)=0$ is subadditive; hence $\gamma\in M$.
Now let $$ t_n=2^{-n}\qquad (n\in\mathbb N), $$
and define
$$\delta_n=t_n^2-t_{n+1}^2. $$
For each $n\in\mathbb N$, let
$$I_n=\left(t_{n+1},\, t_{n+1}+\frac{\delta_n}{2}\right).$$
Since
$$\frac{\delta_n}{2}<t_n-t_{n+1},$$
the intervals $I_n$ are pairwise disjoint and $I_n\subset (t_{n+1},t_n)$.
Define
$$B^c=\bigcup_{n=1}^\infty (I_n\cup(-I_n)),
\qquad B=\mathbb R\setminus B^c .$$
We claim that $0$ is a Lebesgue density point of $B$, but $0\notin \mathcal{D}_\gamma(B)$.
Let
$$m(h)=|B^c\cap(-h,h)| \qquad (h>0). $$
Fix $k\in\mathbb N$ and assume that
$$t_{k+1}<h\le t_k . $$
Then all intervals $I_n$ with $n\ge k+1$ are contained in $(0,h)$, while at most the interval
$I_k$ contributes partially. Hence
$$\sum_{n=k+1}^\infty \delta_n \le m(h)\le \sum_{n=k}^\infty \delta_n .$$
Since the sums telescope,
$$\sum_{n=k+1}^\infty \delta_n=t_{k+1}^2, \qquad
\sum_{n=k}^\infty \delta_n=t_k^2. $$
Therefore
$$t_{k+1}^2\le m(h)\le t_k^2. $$
Because $t_{k+1}<h\le t_k=2t_{k+1}$, we get
$$\frac{h^2}{4}\le m(h)\le 4h^2$$
for all sufficiently small $h>0$.
It follows that
$$0\le \frac{m(h)}{2h}\le 2h \longrightarrow 0 \qquad (h\to0^+). $$
Hence $$ \lim_{h\to0^+}\frac{|B^c\cap(-h,h)|}{2h}=0, $$
so $0$ is a Lebesgue density point of $B$.
On the other hand, for all sufficiently small $h>0$ we have
$$
\frac{\gamma(h^2/4)}{\gamma(2h)}
\le
\frac{\gamma(m(h))}{\gamma(2h)}
\le
\frac{\gamma(4h^2)}{\gamma(2h)}.
$$
Since $\gamma(t)=1/\log(e/t)$ near $0$, for every fixed $c>0$,
$$
\frac{\gamma(ch^2)}{\gamma(2h)}
=
\frac{\log\!\left(\dfrac{e}{2h}\right)}
     {\log\!\left(\dfrac{e}{ch^2}\right)}
\longrightarrow \frac12
\qquad (h\to0^+).
$$
Therefore, by the squeeze theorem,
$$ \lim_{h\to0^+}\frac{\gamma(|B^c\cap(-h,h)|)}{\gamma(2h)}=\frac12\neq 0. $$
Thus $0$ is not a $\gamma$-density point of $B$.

Consequently, the converse of Theorem 1.6 does not hold in general.
\end{example}

\begin{definition}
 \textbf{Condition (A).}
A modulus function $\gamma$ is said to satisfy Condition (A) if for every $\varepsilon>0$
there exist $c_\varepsilon\in(0,1)$ and $\delta_\varepsilon>0$ such that
$$ \frac{\gamma(c_\varepsilon t)}{\gamma(t)}<\varepsilon
\qquad\text{for all }0<t<\delta_\varepsilon. $$

\end{definition}
\begin{theorem}\label{thm3}
Let $A\in \mathcal{L}$ and let $\gamma\in M$ satisfy condition (A). Then  $\Phi(A)\subset\mathcal{D}_\gamma(A)$ holds.
\end{theorem}
\begin{proof}
\textit{Proof.}
Let $x\in \Phi(A)$. Then $x$ is a Lebesgue density point of $A$, so for every $\eta>0$
there exists $\delta_1>0$ such that
$$\frac{|(x-\alpha,x+\alpha)\cap A^c|}{2\alpha}<\eta
\qquad\text{for all }0<\alpha<\delta_1.$$
We want to prove that
$$\lim_{\alpha\to 0^+}
\frac{\gamma(|(x-\alpha,x+\alpha)\cap A^c|)}{\gamma(2\alpha)}=0.$$
Let $\varepsilon>0$ be arbitrary. Since $\gamma$ satisfies Condition (A), there exist
$c\in(0,1)$ and $\delta_2>0$ such that
$$\gamma(ct)<\varepsilon\,\gamma(t) \qquad\text{for all }0<t<\delta_2.$$
Now choose $\delta>0$ such that
$$ 0<\alpha<\delta \quad\Longrightarrow\quad 2\alpha<\delta_2$$
and also
$$\frac{|(x-\alpha,x+\alpha)\cap A^c|}{2\alpha}<c.$$
Then for every $0<\alpha<\delta$ we have
$$|(x-\alpha,x+\alpha)\cap A^c|<2c\alpha=c(2\alpha).$$
Since $\gamma$ is increasing,
$$\gamma(|(x-\alpha,x+\alpha)\cap A^c|) \le \gamma(c(2\alpha)).$$
Hence $$\frac{\gamma(|(x-\alpha,x+\alpha)\cap A^c|)}{\gamma(2\alpha)} \le \frac{\gamma(c(2\alpha))}{\gamma(2\alpha)}
<\varepsilon.$$
Since $\varepsilon>0$ was arbitrary, it follows that
$$\lim_{\alpha\to 0^+} \frac{\gamma(|(x-\alpha,x+\alpha)\cap A^c|)}{\gamma(2\alpha)}=0.$$
Therefore $x\in \mathcal{D}_\gamma(A)$, and so $\Phi(A)\subset \mathcal{D}_\gamma(A)$.

\end{proof}

\begin{theorem}\label{m3}
Let $A\in\mathcal L$ and let $\gamma\in M$ satisfy condition (A). Then
$$\mathcal{D}_\gamma(A)=\Phi(A).$$
Consequently,
$$|A\Delta \mathcal{D}_\gamma(A)|=0.$$
\end{theorem} 
 \begin{proof}   
By Theorem \ref{thm6}, every $\gamma$-density point of $A$ is a Lebesgue density point of $A$. Hence
$$\mathcal{D}_\gamma(A)\subset \Phi(A).$$
On the other hand, by Theorem \ref{thm3}, since $\gamma$ satisfies Condition (A), every Lebesgue density point of $A$ is a $\gamma$-density point of $A$. Therefore
$$\Phi(A)\subset \mathcal{D}_\gamma(A).$$
Thus,
$$\mathcal{D}_\gamma(A)=\Phi(A).$$
Now, by the Lebesgue Density Theorem,
$$|A\Delta \Phi(A)|=0.$$
Since $\mathcal{D}_\gamma(A)=\Phi(A)$, it follows that
$$|A\Delta \mathcal{D}_\gamma(A)|=0.$$

\end{proof}
\begin{proposition}\label{pn1.11}
Let $A \in L$, let $\gamma \in \mathcal{M}$, and let $x \in \mathbb{R}$. Then the following statements are equivalent:
\begin{enumerate}
\item $
\lim_{\alpha \to 0^+}
\frac{\gamma(|(x-\alpha,x+\alpha)\cap A^c|)}{\gamma(2\alpha)}=0.$
\item
$\lim_{n \to \infty}
\frac{\gamma(|(x-\frac1n,x+\frac1n)\cap A^c|)}{\gamma(2/n)}=0.$
\end{enumerate}
\end{proposition}
\begin{proof}
$(1)\Rightarrow(2)$ is immediate by taking $\alpha=\frac1n$.
Now assume that $(2)$ holds. Let $\alpha>0$ be arbitrary and choose $n\in\mathbb{N}$ such that
$$\frac{1}{n+1}\le \alpha \le \frac{1}{n}.$$
Then$$(x-\alpha,x+\alpha)\cap A^c \subset (x-1/n,x+1/n)\cap A^c,$$
hence, since $\gamma$ is increasing,
$$\gamma(|(x-\alpha,x+\alpha)\cap A^c|)\le \gamma(|(x-1/n,x+1/n)\cap A^c|).$$
Also, from $$\alpha \ge \frac{1}{n+1}$$
and the monotonicity of $\gamma$, we get
$$\gamma(2\alpha)\ge \gamma\!\left(\frac{2}{n+1}\right).$$
Therefore
$$\frac{\gamma(|(x-\alpha,x+\alpha)\cap A^c|)}{\gamma(2\alpha)}
\le \frac{\gamma(|(x-1/n,x+1/n)\cap A^c|)}{\gamma(2/(n+1))}.$$
Next, using subadditivity and monotonicity of $\gamma$, we have
$$\frac{2}{n} = \frac{2}{n+1}+\frac{2}{n(n+1)}, $$
so
$$ \gamma\!\left(\frac{2}{n}\right) \le \gamma\!\left(\frac{2}{n+1}\right) + \gamma\!\left(\frac{2}{n(n+1)}\right).$$
Since $$\frac{2}{n(n+1)} \le \frac{2}{n+1}, $$
we obtain
$$\gamma\!\left(\frac{2}{n(n+1)}\right)\le \gamma\!\left(\frac{2}{n+1}\right), $$
and hence $$ \gamma\!\left(\frac{2}{n}\right)\le 2\,\gamma\!\left(\frac{2}{n+1}\right). $$
Thus
$$\frac{1}{\gamma(2/(n+1))}\le \frac{2}{\gamma(2/n)}. $$
Combining this with the previous estimate yields
$$ \frac{\gamma(|(x-\alpha,x+\alpha)\cap A^c|)}{\gamma(2\alpha)}
\le 2\,\frac{\gamma(|(x-1/n,x+1/n)\cap A^c|)}{\gamma(2/n)}. $$
As $\alpha\to0^+$, necessarily $n\to\infty$, and by assumption $(2)$ the right-hand side tends to $0$. Therefore
$$\lim_{\alpha\to0^+}
\frac{\gamma(|(x-\alpha,x+\alpha)\cap A^c|)}{\gamma(2\alpha)}=0. $$
So $(1)$ holds.
Hence $(1)$ and $(2)$ are equivalent.
\end{proof}
\begin{proposition}
For any $A \in L$ and any modulus function $\gamma \in \mathcal{M}$, the sets
$\mathcal{D}_\gamma(A)$ $, \mathcal{D}_\gamma^+(A)$ and  $\mathcal{D}_\gamma^-(A)$
are measurable subsets of $\mathbb{R}$. In fact, each of them is an $F_{\sigma\delta}$ set.
\end{proposition}
\begin{proof}
We prove the claim only for $\mathcal{D}_\gamma(A)$, since the proofs for $\mathcal{D}_\gamma^+(A)$ and $\mathcal{D}_\gamma^-(A)$ are analogous.
For each $m,n \in \mathbb{N}$, define
$$ F_{m,n} := \left\{ x \in \mathbb{R} : \frac{\gamma\bigl(|(x-\frac1n,x+\frac1n)\cap A^c|\bigr)}{\gamma(2/n)} \le \frac1m \right\}. $$
By Proposition \ref{pn1.11}, for every $x \in \mathbb{R}$,
$$ x \in \mathcal{D}_\gamma(A) \quad\Longleftrightarrow\quad \lim_{n\to\infty}
\frac{\gamma\bigl(|(x-\frac1n,x+\frac1n)\cap A^c|\bigr)}{\gamma(2/n)} =0. $$
Therefore, $$\mathcal{D}_\gamma(A) = \bigcap_{m\in\mathbb{N}} \bigcup_{N\in\mathbb{N}} \bigcap_{n\ge N} F_{m,n}. $$
So it remains to show that each $F_{m,n}$ is closed.
Fix $n\in\mathbb{N}$ and define $$ f_n(x) :=
\left|(x-\tfrac1n,x+\tfrac1n)\cap A^c\right|. $$ We claim that $f_n$ is continuous. Indeed, for any $x,y\in\mathbb{R}$, $$ |f_n(x)-f_n(y)| \le \left| (x-\tfrac1n,x+\tfrac1n)\,\Delta\,(y-\tfrac1n,y+\tfrac1n) \right| \le 2|x-y|. $$
Hence $f_n$ is Lipschitz continuous.
Next, we show that $\gamma$ is continuous on $[0,\infty)$. Since $\gamma$ is a modulus function, for $a\ge b\ge0$ we have $$ \gamma(a)-\gamma(b)\le \gamma(a-b). $$
Thus, for any $s,t\ge0$, $$ |\gamma(s)-\gamma(t)| \le \gamma(|s-t|). $$
Because $\gamma$ is right-continuous at $0$ and increasing, we have $$ \gamma(h)\to 0 \qquad (h\to 0^+), $$
so the above inequality implies that $\gamma$ is continuous on $[0,\infty)$.
Therefore the function
$$ x \mapsto \frac{\gamma(f_n(x))}{\gamma(2/n)} $$
is continuous for each $n\in\mathbb{N}$. It follows that $$ F_{m,n} = \left\{ x \in \mathbb{R} :
\frac{\gamma(f_n(x))}{\gamma(2/n)} \le \frac1m \right\} $$ is closed.
Consequently,
$$ \mathcal{D}_\gamma(A) = \bigcap_{m\in\mathbb{N}} \bigcup_{N\in\mathbb{N}} \bigcap_{n\ge N} F_{m,n} $$
is an $F_{\sigma\delta}$ set, and in particular it is measurable.
The same argument applies to $\mathcal{D}_\gamma^+(A)$ and $\mathcal{D}_\gamma^-(A)$.
\end{proof}
\begin{proposition}\label{m13}
For any set $A \in L$, any modulus function $\gamma \in \mathcal{M}$, and any $z \in \mathbb{R}$, we have
$$\mathcal{D}_\gamma(A+z)=\mathcal{D}_\gamma(A)+z,$$
where$$ A+z:=\{a+z:a\in A\}. $$
\end{proposition}

\begin{proof}
Let $x \in \mathbb{R}$. Then $$ x \in \mathcal{D}_\gamma(A+z) $$ if and only if $$ \lim_{\alpha \to 0^+} \frac{\gamma\bigl(|(x-\alpha,x+\alpha)\cap (A+z)^c|\bigr)}{\gamma(2\alpha)}=0. $$
Since
$$ (A+z)^c=A^c+z, $$ we have $$ (x-\alpha,x+\alpha)\cap (A+z)^c = \bigl((x-z)-\alpha,(x z)+\alpha\bigr)\cap A^c + z. $$ Therefore, by translation invariance of Lebesgue measure,
$$ \left|(x-\alpha,x+\alpha)\cap (A+z)^c\right| = \left|\bigl((x-z)-\alpha,(x-z)+\alpha\bigr)\cap A^c\right|. $$
Hence $$ \lim_{\alpha \to 0^+} \frac{\gamma\bigl(|(x-\alpha,x+\alpha)\cap (A+z)^c|\bigr)}{\gamma(2\alpha)} = \lim_{\alpha \to 0^+}
\frac{\gamma\bigl(|((x-z)-\alpha,(x-z)+\alpha)\cap A^c|\bigr)}{\gamma(2\alpha)}. $$
Thus, $$ x \in \mathcal{D}_\gamma(A+z) \quad\Longleftrightarrow\quad x-z \in \mathcal{D}_\gamma(A). $$
This is equivalent to $$ x \in \mathcal{D}_\gamma(A)+z. $$
Therefore, $$ \mathcal{D}_\gamma(A+z)=\mathcal{D}_\gamma(A)+z. $$
\end{proof}

\begin{theorem}\label{thm2}
Let $A,B \in L$ and let $\gamma \in \mathcal{M}$. Then the following statements hold:

\begin{enumerate}
\item
$\mathcal{D}_\gamma(\varnothing)=\varnothing \qquad\text{and}\qquad \mathcal{D}_\gamma(\mathbb{R})=\mathbb{R}.$

\item If $A \subset B$, then $ \mathcal{D}_\gamma(A)\subset \mathcal{D}_\gamma(B).$
\item
If $A \sim B$, then $ \mathcal{D}_\gamma(A)=\mathcal{D}_\gamma(B).$

\item $ \mathcal{D}_\gamma(A\cap B)=\mathcal{D}_\gamma(A)\cap \mathcal{D}_\gamma(B).$
\end{enumerate}
\end{theorem}

\begin{proof}
\begin{enumerate}
\item
For every $x\in\mathbb{R}$ and every $\alpha>0$, we have $ (x-\alpha,x+\alpha)\cap \mathbb{R}^c=\varnothing, $
hence $$ \frac{\gamma\bigl(|(x-\alpha,x+\alpha)\cap \mathbb{R}^c|\bigr)}{\gamma(2\alpha)} = \frac{\gamma(0)}{\gamma(2\alpha)} =0.$$
Therefore $x\in \mathcal{D}_\gamma(\mathbb{R})$ for every $x\in\mathbb{R}$, so
$ \mathcal{D}_\gamma(\mathbb{R})=\mathbb{R}. $

On the other hand, for every $x\in\mathbb{R}$ and every $\alpha>0$,
$$(x-\alpha,x+\alpha)\cap \varnothing^c=(x-\alpha,x+\alpha), $$
so
$$ \frac{\gamma\bigl(|(x-\alpha,x+\alpha)\cap \varnothing^c|\bigr)}{\gamma(2\alpha)} = \frac{\gamma(2\alpha)}{\gamma(2\alpha)} =1. $$
Hence no point belongs to $\mathcal{D}_\gamma(\varnothing)$, and thus $ \mathcal{D}_\gamma(\varnothing)=\varnothing. $

\item Assume that $A\subset B$ and let $x\in \mathcal{D}_\gamma(A)$. Then $ B^c\subset A^c, $ and therefore for every $\alpha>0$, $$ |(x-\alpha,x+\alpha)\cap B^c| \le |(x-\alpha,x+\alpha)\cap A^c|.$$
Since $\gamma$ is increasing, $$ \gamma\bigl(|(x-\alpha,x+\alpha)\cap B^c|\bigr) \le \gamma\bigl(|(x-\alpha,x+\alpha)\cap A^c|\bigr).$$
Dividing by $\gamma(2\alpha)$ and passing to the limit as $\alpha\to0^+$, we obtain
$$ \lim_{\alpha\to0^+} \frac{\gamma\bigl(|(x-\alpha,x+\alpha)\cap B^c|\bigr)}{\gamma(2\alpha)} =0. $$
Thus $x\in \mathcal{D}_\gamma(B)$, and so $ \mathcal{D}_\gamma(A)\subset \mathcal{D}_\gamma(B). $

\item
Assume that $A\sim B$, that is, $ |A\Delta B|=0. $
Then also $|A^c\Delta B^c|=|A\Delta B|=0.$
Hence, for every $\alpha>0$,
$$ \bigl|((x-\alpha,x+\alpha)\cap A^c)\Delta((x-\alpha,x+\alpha)\cap B^c)\bigr|=0, $$
which implies $$ |(x-\alpha,x+\alpha)\cap A^c| = |(x-\alpha,x+\alpha)\cap B^c|. $$
Therefore, for every $x\in\mathbb{R}$, $$ \frac{\gamma\bigl(|(x-\alpha,x+\alpha)\cap A^c|\bigr)}{\gamma(2\alpha)} = \frac{\gamma\bigl(|(x-\alpha,x+\alpha)\cap B^c|\bigr)}{\gamma(2\alpha)}. $$
Passing to the limit, we get
$$ x\in \mathcal{D}_\gamma(A) \quad\Longleftrightarrow\quad x\in \mathcal{D}_\gamma(B).$$
Hence $ \mathcal{D}_\gamma(A)=\mathcal{D}_\gamma(B). $

\item
Since $ A\cap B\subset A \qquad\text{and}\qquad A\cap B\subset B,$
part (2) gives
$$
\mathcal{D}_\gamma(A\cap B)\subset \mathcal{D}_\gamma(A)\cap \mathcal{D}_\gamma(B).
$$

For the reverse inclusion, let $x\in \mathcal{D}_\gamma(A)\cap \mathcal{D}_\gamma(B)$. Put $ I_\alpha:=(x-\alpha,x+\alpha).$
Then
$$
\lim_{\alpha\to0^+}
\frac{\gamma\bigl(|I_\alpha\cap A^c|\bigr)}{\gamma(2\alpha)}=0
\qquad\text{and}\qquad
\lim_{\alpha\to0^+}
\frac{\gamma\bigl(|I_\alpha\cap B^c|\bigr)}{\gamma(2\alpha)}=0.
$$
Now $ (A\cap B)^c=A^c\cup B^c, $
so
$$
I_\alpha\cap (A\cap B)^c
=
(I_\alpha\cap A^c)\cup(I_\alpha\cap B^c).
$$
Hence
$$
|I_\alpha\cap (A\cap B)^c|
\le
|I_\alpha\cap A^c|+|I_\alpha\cap B^c|.
$$
Using the monotonicity and subadditivity of $\gamma$, we obtain
$$
\gamma\bigl(|I_\alpha\cap (A\cap B)^c|\bigr)
\le
\gamma\bigl(|I_\alpha\cap A^c|+|I_\alpha\cap B^c|\bigr)
\le
\gamma\bigl(|I_\alpha\cap A^c|\bigr)+\gamma\bigl(|I_\alpha\cap B^c|\bigr).
$$
Therefore
$$
0\le
\frac{\gamma\bigl(|I_\alpha\cap (A\cap B)^c|\bigr)}{\gamma(2\alpha)}
\le
\frac{\gamma\bigl(|I_\alpha\cap A^c|\bigr)}{\gamma(2\alpha)}
+
\frac{\gamma\bigl(|I_\alpha\cap B^c|\bigr)}{\gamma(2\alpha)}.
$$
Passing to the limit as $\alpha\to0^+$, we get
$$
\lim_{\alpha\to0^+}
\frac{\gamma\bigl(|I_\alpha\cap (A\cap B)^c|\bigr)}{\gamma(2\alpha)}
=0.
$$ Thus $x\in \mathcal{D}_\gamma(A\cap B)$, and so $$ \mathcal{D}_\gamma(A)\cap \mathcal{D}_\gamma(B)\subset \mathcal{D}_\gamma(A\cap B).$$
 
Combining both inclusions, we conclude that
$$ \mathcal{D}_\gamma(A\cap B)=\mathcal{D}_\gamma(A)\cap \mathcal{D}_\gamma(B).$$
\end{enumerate}
\end{proof}

\begin{corollary}\label{cor1.15}
    If the modulus function $\gamma$ satisfies Condition (A)., then the operator $\mathcal{D}_\gamma$ is a lower-density operator.
\end{corollary}
\begin{proof}
    This result is an immediate consequence of Theorems \ref{m3} and \ref{thm2}. 
\end{proof}
The notion of a $\psi$-density point of a set was introduced in \cite{mepsi}. In the following theorem, we compare this concept with the notion of a $\gamma$-density point. More precisely, under suitable assumptions on $\psi$, we show that $\psi$ induces a modulus function $\gamma$, and that every $\psi$-density point of a set is also a $\gamma$-density point.
\begin{theorem}\label{thm15}
Let $\psi \in C$ be subadditive, and define $\gamma:[0,\infty)\to[0,\infty)$ by
$ \gamma(0)=0,\qquad \gamma(t)=\psi(t)\quad (t>0). $
Then $\gamma \in \mathcal{M}$.
Moreover, if
$$ \limsup_{t\to 0^+}\frac{\psi(t)}{t}<\infty, $$
then every $\psi$-density point of $A$ is a $\gamma$-density point of $A$.
\end{theorem}
\begin{proof}
First we show that $\gamma \in \mathcal{M}$.

Since $\psi \in C$, the function $\psi$ is continuous, nondecreasing on $(0,\infty)$, and
$$
\lim_{t\to 0^+}\psi(t)=0.
$$
By the definition of $\gamma$,
$$
\gamma(0)=0,\qquad \gamma(t)=\psi(t)\quad (t>0).
$$
Hence $\gamma$ is right-continuous at $0$. Also, since $\psi$ is nondecreasing on $(0,\infty)$, $\gamma$ is increasing on $[0,\infty)$.

Next, for $a,b>0$, the subadditivity of $\psi$ gives
$$
\gamma(a+b)=\psi(a+b)\le \psi(a)+\psi(b)=\gamma(a)+\gamma(b).
$$
If one of $a,b$ is equal to $0$, the inequality is trivial. Thus $\gamma$ is subadditive on $[0,\infty)$.

Finally, for every $t>0$ we have $\gamma(t)=\psi(t)>0$, while $\gamma(0)=0$. Therefore
$$
\gamma(t)=0 \iff t=0.
$$
So $\gamma$ satisfies all the conditions of a modulus function, and hence $\gamma \in \mathcal{M}$.

Now let $x$ be a $\psi$-density point of $A$. Then, by definition,
$$
\lim_{h\to 0^+}\frac{|A^c\cap [x-h,x+h]|}{2h\,\psi(2h)}=0.
$$
Since endpoints do not affect Lebesgue measure, this is equivalent to
$$
\lim_{h\to 0^+}\frac{|A^c\cap (x-h,x+h)|}{2h\,\psi(2h)}=0.
$$
$$
|(x-h,x+h)\cap A^c|\le 2h
$$
for every $h>0$, and the above limit becomes
$$
\lim_{h\to 0^+}\frac{|(x-h,x+h)\cap A^c|}{2h\,\psi(2h)}=0.
$$

Assume now that
$$
\limsup_{t\to 0^+}\frac{\psi(t)}{t}<\infty.
$$
Then there exist constants $C>0$ and $\delta>0$ such that
$$
\psi(t)\le Ct \qquad \text{for all } 0<t<\delta.
$$
For $0<h<\delta/2$, we have $|(x-h,x+h)\cap A^c|\le 2h<\delta$, and therefore
$$
\gamma(|(x-h,x+h)\cap A^c|)=\psi(|(x-h,x+h)\cap A^c|)\le C|(x-h,x+h)\cap A^c|.
$$
Hence
$$
\frac{\gamma(|(x-h,x+h)\cap A^c|)}{\gamma(2h)}
=
\frac{\psi(|(x-h,x+h)\cap A^c|)}{\psi(2h)}
$$ $$ \le 
C\,\frac{|(x-h,x+h)\cap A^c|}{\psi(2h)}
=
2Ch\cdot \frac{|(x-h,x+h)\cap A^c|}{2h\,\psi(2h)}.
$$
As $h\to 0^+$, we have
$$
2Ch\to 0
$$
and
$$
\frac{|(x-h,x+h)\cap A^c|}{2h\,\psi(2h)}\to 0.
$$
Therefore,
$$
\frac{\gamma(|(x-h,x+h)\cap A^c|)}{\gamma(2h)}\to 0.
$$
That is,
$$
\lim_{h\to 0^+}\frac{\gamma(|(x-h,x+h)\cap A^c|)}{\gamma(2h)}=0.
$$
Thus $x$ is a $\gamma$-density point of $A$.

Hence every $\psi$-density point of $A$ is a $\gamma$-density point of $A$.
\end{proof}

\section{Modulus Density topology($\gamma$-Density topology)}
The idea of a new density-type topology, which is defined with the help of the $\gamma$-density operator, is presented in this section.

\begin{theorem}
For any $A\in \mathcal{L}$ and $\gamma \in \mathcal{M}$, $$\tau_{\gamma}:= \{A\in \mathcal{L}: A \subset \mathcal{D}_{\gamma}(A)\}$$ is a topology on $\mathbb{R}$.
\end{theorem}
\begin{proof}
By Theorem \ref{thm2}, we already have  $\emptyset \in \tau_{\gamma}$ and $\mathbb{R}\in \tau_{\gamma}$. Let $A$ and $B$ be an arbitrary sets in $\tau_{\gamma}$ and let
    $x\in A\cap B$ be an arbitrary point. Then, $x\in A \subset \mathcal{D}_{\gamma}(A)$ and $x\in B \subset \mathcal{D}_{\gamma}(B)$. So, 
   $$x\in \mathcal{D}_{\gamma}(A) \cap \mathcal{D}_{\gamma}(A)= \mathcal{D}_{\gamma}(A\cap B)$$ implies that $A\cap B \subset \mathcal{D}_{\gamma}(A\cap B).$  Hence, $A\cap B \in \tau_{\gamma}.$

   Let $\{A_{i}\}_{i\in \lambda}$ be a collection of $\gamma$-open subset of $\mathbb{R}$. Let $B$ be a measurable kernel of the set $\cup_{i\in \lambda} A_{i}$,(Such B exists because the $(\mathbb{R}, \mathcal{L}, \mathcal{I})$ where $\mathcal{I}$ is the sigma ideal of null sets has the Hull property.) For every  ${i\in \lambda}$ we have, $$|(A_{i}\cap B)\Delta A_i|=0$$ and $$B\subset \cup_{i\in \lambda} A_{i} \subset \cup_{i\in \lambda} \mathcal{D}_\gamma (A_{i})\subset \cup_{i\in \lambda} \mathcal{D}_\gamma (A_{i} \cap B)\subset \mathcal{D}_\gamma ( B). $$ Since $|\mathcal{D}_\gamma (B)-B|=0,$ $\cup_{i\in \lambda} A_{i} \in \mathcal{L}.$ Furhtermore let $x\in \bigcup_{i\in \lambda} A_i$ then $x\in A_j$ for some $j\in \lambda$. Since $A_j$ is $\gamma$-open set and $A_j \subset \bigcup_{i\in \lambda} A_i$, then by $(ii)$ of Theorem \ref{thm2} we have  $$\mathcal{D}_{\gamma}(A_j)\subset \mathcal{D}_{\gamma}(\bigcup_{i\in \lambda} A_i).$$ Hence, $x\in \mathcal{D}_{\gamma}(\bigcup_{i\in \lambda}A_i)$ and this proves the theorem.
\end{proof}
\begin{theorem}
Let $\gamma_1,\gamma_2\in M$. Assume that there exist constants $a,b,\delta>0$ such that
$$ a\leq \frac{\gamma_1(t)}{\gamma_2(t)}\leq b $$
for all $0<t<\delta$. Then
$$ \tau_{\gamma_1}=\tau_{\gamma_2}. $$
\end{theorem}

\begin{proof}
By Proposition \ref{pn1.5}, for every $A\in L$ we have $ D_{\gamma_1}(A)=D_{\gamma_2}(A). $
Now let $A\in \tau_{\gamma_1}$. Then $ A\in L \quad \text{and} \quad A\subset D_{\gamma_1}(A). $
Since $D_{\gamma_1}(A)=D_{\gamma_2}(A)$, it follows that $ A\subset D_{\gamma_2}(A). $
Hence $A\in \tau_{\gamma_2}$. Therefore, $ \tau_{\gamma_1}\subset \tau_{\gamma_2}. $
Similarly, if $A\in \tau_{\gamma_2}$, then
$ A\in L \quad \text{and} \quad A\subset D_{\gamma_2}(A)=D_{\gamma_1}(A), $
so $A\in \tau_{\gamma_1}$. Thus, $ \tau_{\gamma_2}\subset \tau_{\gamma_1}.$
Consequently, $ \tau_{\gamma_1}=\tau_{\gamma_2}. $
\end{proof}

\begin{definition}\label{def4}
A set $V\subset \mathbb{R}$ is $\tau_{\gamma}$-neighborhood of a point $x\in \mathbb{R}$ if there exists a set $A_x\in \tau_\gamma$ such that $x\in A_x$ and $A_x\subset V$. 
\end{definition}
\begin{proposition}
 A set $V$ is $\tau_\gamma$-open if and only if $V$ is $\tau_\gamma$-neighborhood of each of its points. 
\end{proposition}
\begin{proof}
   The result follows directly from Definition \ref{def4}, and the proof is omitted for brevity. 
\end{proof}
\begin{theorem}\label{thm2.6}
    The Lebesgue density topology is finer than the modulus density topology, i.e., $$\tau_\gamma \subset \tau_d.$$
\end{theorem}

\begin{proof}
    Let $A\in \tau_\gamma$, then $A\subset \mathcal{D}_\gamma (A).$ By Theorem \ref{thm6}, $A\subset \Phi(A)$ so $A\in \tau_d$.
\end{proof}
\begin{theorem}\label{thm2.5}
     If the modulus function $\gamma$ has the condition A., then the topology $\tau_{\gamma}=\tau_d.$ 
\end{theorem}
\begin{proof}
From theorem \ref{thm2.6}, $\tau_\gamma \subset \tau_d.$
Assume that $A\in \tau_d$, $A\subset \Phi(A)$. By Theorem \ref{thm3} $\Phi(A)\subset \mathcal{D}_\gamma(A)$, we have $A\in \tau_\gamma.$ 
\end{proof}

\begin{example}
Consider the modulus function $\gamma$ given in Example \ref{ex1.7}, and let $B \subset \mathbb{R}$ be the set constructed there. By Example \ref{ex1.7}, we have $ 0 \in \Phi(B) \setminus D_{\gamma}(B).$
Now define $ U:=\Phi(B). $
We claim that $U$ is open in the Lebesgue density topology, but it is not open in the $\gamma$-density topology.
First, $U \in \tau_d$. Indeed, by the definition of the Lebesgue density topology,
$\tau_d=\{A \in L : A \subset \Phi(A)\}.$
Since $U=\Phi(B)$, it follows from the idempotence of the density operator that $ \Phi(U)=\Phi(\Phi(B))=\Phi(B)=U, $
and hence $ U \subset \Phi(U).$
Therefore $U \in \tau_d$.

On the other hand, by the Lebesgue Density Theorem,
$ B \sim \Phi(B)=U. $
Hence, by Theorem \ref{thm2}(iii),
$D_{\gamma}(U)=D_{\gamma}(B).$ Since $0 \in \Phi(B)=U$ and $0 \notin D_{\gamma}(B)$, we obtain
$ 0 \notin D_{\gamma}(U). $
Thus $ 0 \in U \setminus D_{\gamma}(U), $
so $ U \not\subset D_{\gamma}(U).$
Therefore $U \notin \tau_{\gamma}$.

Consequently, $U$ is open in the Lebesgue density topology, but it is not open in the $\gamma$-density topology. This shows that if $\gamma$ does not satisfy Condition (A), then $\tau_d$ and $\tau_{\gamma}$ need not coincide.
\end{example}
\begin{theorem}
Under the assumptions of Theorem \ref{thm15}, the $\gamma$-density topology is finer than the $\psi$-density topology.
\end{theorem}
\begin{proof}
    Assume that $A\in \tau_\psi$, $A\subset \Phi_\psi(A)$. By Theorem \ref{thm15} $\Phi_\psi(A)\subset \mathcal{D}_\gamma(A)$, we have $A\in \tau_\gamma.$ 
\end{proof}
\begin{theorem}\label{thm2.10}
Every countable subset of $\mathbb{R}$ is a closed set in $\gamma$-density topology. Consequently, modulus density topology is not separable, and every compact subset in modulus density topology is finite.
\end{theorem}
\begin{proof}
It is sufficient to prove that the complement of any countable set is $\gamma$-open. Let $C$ be a countable set and let $x \in C^c$. Then
$$
|(x-\alpha,x+\alpha)\cap C|=0
$$
for every $\alpha>0$, since every countable set has Lebesgue measure zero. Hence
$$
\lim_{\alpha\to 0^+}\frac{\gamma(|(x-\alpha,x+\alpha)\cap C|)}{\gamma(2\alpha)}
=
\lim_{\alpha\to 0^+}\frac{\gamma(0)}{\gamma(2\alpha)}
=0.
$$
Therefore $x\in \mathcal{D}_\gamma(C^c)$, and so
$$
C^c \subset \mathcal{D}_\gamma(C^c).
$$
This shows that $C^c\in\tau_\gamma$, and consequently $C$ is $\tau_\gamma$-closed.

Now we show that $(\mathbb{R},\tau_\gamma)$ is not separable. Suppose, on the contrary, that there exists a countable dense set $D \subset \mathbb{R}$. Since $D$ is countable, it is $\tau_\gamma$-closed by the first part. Hence
$$
\overline{D}^{\,\tau_\gamma}=D.
$$
But $D$ is dense, so $\overline{D}^{\,\tau_\gamma}=\mathbb{R}$. Therefore $D=\mathbb{R}$, which is impossible since $\mathbb{R}$ is uncountable. Thus $(\mathbb{R},\tau_\gamma)$ is not separable.

Finally, let $K$ be a compact subset of $(\mathbb{R},\tau_\gamma)$. Assume that $K$ is infinite. Then $K$ contains a countably infinite subset $C$. Since $C$ is countable, it is $\tau_\gamma$-closed in $\mathbb{R}$, and therefore it is closed in the subspace $K$. Hence $C$ is compact as a closed subset of a compact set.

On the other hand, every subset of $C$ is countable, hence $\tau_\gamma$-closed in $\mathbb{R}$. It follows that every subset of $C$ is closed in the subspace topology on $C$. Therefore every subset of $C$ is also open, so $C$ is a discrete space. But an infinite discrete space is not compact, because the open cover
$$
\big\{\{x\}:x\in C\big\}
$$
has no finite subcover. This is a contradiction. Therefore $K$ must be finite.
\end{proof}

\begin{theorem}
Let $A \subset \mathbb{R}$ and let $x \in \mathbb{R}$. Then $x$ is a $\tau_{\gamma}$-limit point of $A$ if and only if
$$
\limsup_{\alpha \to 0^+}
\frac{\gamma\bigl(m^*((x-\alpha,x+\alpha)\cap (A\setminus\{x\}))\bigr)}{\gamma(2\alpha)}
>0,
$$
where $m^*$ denotes the Lebesgue outer measure.

Equivalently, since $m^*(\{x\})=0$, $x$ is a $\tau_{\gamma}$-limit point of $A$ if and only if
$$
\limsup_{\alpha \to 0^+}
\frac{\gamma\bigl(m^*((x-\alpha,x+\alpha)\cap A)\bigr)}{\gamma(2\alpha)}
>0.
$$
\end{theorem}

\begin{proof}
First assume that $x$ is a $\tau_{\gamma}$-limit point of $A$. We show that
$$
\limsup_{\alpha \to 0^+}
\frac{\gamma\bigl(m^*((x-\alpha,x+\alpha)\cap (A\setminus\{x\}))\bigr)}{\gamma(2\alpha)}
>0.
$$
Suppose, to the contrary, that
$$
\limsup_{\alpha \to 0^+}
\frac{\gamma\bigl(m^*((x-\alpha,x+\alpha)\cap (A\setminus\{x\}))\bigr)}{\gamma(2\alpha)}
=0.
$$
Then in fact
$$
\lim_{\alpha \to 0^+}
\frac{\gamma\bigl(m^*((x-\alpha,x+\alpha)\cap (A\setminus\{x\}))\bigr)}{\gamma(2\alpha)}
=0.
$$
Define
$$
U:=\mathbb{R}\setminus (A\setminus\{x\}).
$$
Then $x\in U$. We claim that $U\in\tau_{\gamma}$.

Let $z\in U$. If $z\neq x$, then $z\notin A$, so there exists $\delta>0$ such that
$$
(z-\delta,z+\delta)\cap (A\setminus\{x\})=\varnothing.
$$
Hence, for all sufficiently small $r>0$,
$$
(z-r,z+r)\subset U,
$$
and therefore
$$
\frac{\gamma(|(z-r,z+r)\cap U^c|)}{\gamma(2r)}=0.
$$
Thus $z\in D_{\gamma}(U)$.

Now consider the point $x$. Since
$$
U^c=A\setminus\{x\},
$$
we have
$$
\frac{\gamma\bigl(|(x-\alpha,x+\alpha)\cap U^c|\bigr)}{\gamma(2\alpha)}
=
\frac{\gamma\bigl(m^*((x-\alpha,x+\alpha)\cap (A\setminus\{x\}))\bigr)}{\gamma(2\alpha)}
\to 0.
$$
Hence $x\in D_{\gamma}(U)$. Therefore
$$
U\subset D_{\gamma}(U),
$$
so $U\in\tau_{\gamma}$.

But
$$
U\cap (A\setminus\{x\})=\varnothing,
$$
which means that $U$ is a $\tau_{\gamma}$-neighborhood of $x$ disjoint from $A\setminus\{x\}$. This contradicts the assumption that $x$ is a $\tau_{\gamma}$-limit point of $A$. Therefore
$$
\limsup_{\alpha \to 0^+}
\frac{\gamma\bigl(m^*((x-\alpha,x+\alpha)\cap (A\setminus\{x\}))\bigr)}{\gamma(2\alpha)}
>0.
$$

Conversely, assume that
$$
\limsup_{\alpha \to 0^+}
\frac{\gamma\bigl(m^*((x-\alpha,x+\alpha)\cap (A\setminus\{x\}))\bigr)}{\gamma(2\alpha)}
>0.
$$
We show that $x$ is a $\tau_{\gamma}$-limit point of $A$.

Suppose not. Then there exists a set $G\in\tau_{\gamma}$ such that
$$
x\in G
\qquad\text{and}\qquad
G\cap (A\setminus\{x\})=\varnothing.
$$
Hence
$$
A\setminus\{x\}\subset G^c.
$$
Therefore, for every $\alpha>0$,
$$
m^*((x-\alpha,x+\alpha)\cap (A\setminus\{x\}))
\le
m^*((x-\alpha,x+\alpha)\cap G^c).
$$
Since $G^c$ is measurable, the right-hand side equals
$$
|(x-\alpha,x+\alpha)\cap G^c|.
$$
Using the monotonicity of $\gamma$, we obtain
$$
\frac{\gamma\bigl(m^*((x-\alpha,x+\alpha)\cap (A\setminus\{x\}))\bigr)}{\gamma(2\alpha)}
\le
\frac{\gamma\bigl(|(x-\alpha,x+\alpha)\cap G^c|\bigr)}{\gamma(2\alpha)}.
$$
Now $G\in\tau_{\gamma}$ and $x\in G$, so $x\in D_{\gamma}(G)$. By definition,
$$
\lim_{\alpha\to 0^+}
\frac{\gamma\bigl(|(x-\alpha,x+\alpha)\cap G^c|\bigr)}{\gamma(2\alpha)}
=0.
$$
Consequently,
$$
\limsup_{\alpha \to 0^+}
\frac{\gamma\bigl(m^*((x-\alpha,x+\alpha)\cap (A\setminus\{x\}))\bigr)}{\gamma(2\alpha)}
=0,
$$
which contradicts the assumption. Hence every $\tau_{\gamma}$-neighborhood of $x$ meets $A\setminus\{x\}$, and so $x$ is a $\tau_{\gamma}$-limit point of $A$.
\end{proof}

\begin{theorem}\label{9}
Assume that $D_{\gamma}$ is a lower-density operator on $L$ (in particular, this holds whenever $\gamma$ satisfies Condition $(A)$). Let $A \subset \mathbb{R}$, and let $K \in L$ be a measurable kernel of $A$, that is,
$$
K \subset A
\qquad\text{and}\qquad
|A\setminus K|=0.
$$
Then
$$
\operatorname{Int}_{\tau_{\gamma}}(A)=A\cap D_{\gamma}(K).
$$
\end{theorem}

\begin{proof}
First let $x \in \operatorname{Int}_{\tau_{\gamma}}(A)$. Then there exists a set $U \in \tau_{\gamma}$ such that
$$
x \in U \subset A.
$$
Since $U \subset A$ and $K \subset A$ with $|A\setminus K|=0$, we have
$$
U\setminus K \subset A\setminus K,
$$
hence
$$
|U\setminus K|=0.
$$
Therefore
$$
U\triangle (U\cap K)=U\setminus K
$$
is a null set, so
$$
U \sim U\cap K.
$$
By Theorem 1.14(iii),
$$
D_{\gamma}(U)=D_{\gamma}(U\cap K).
$$
Since $U\in\tau_{\gamma}$, we have $U\subset D_{\gamma}(U)$, and therefore
$$
x\in U \subset D_{\gamma}(U)=D_{\gamma}(U\cap K)\subset D_{\gamma}(K).
$$
Also $x\in A$, because $U\subset A$. Hence
$$
x\in A\cap D_{\gamma}(K).
$$
This proves that
$$
\operatorname{Int}_{\tau_{\gamma}}(A)\subset A\cap D_{\gamma}(K).
$$

Now let $x\in A\cap D_{\gamma}(K)$. We distinguish two cases.

If $x\in K$, put
$$
H:=K.
$$
If $x\in A\setminus K$, put
$$
H:=K\cup\{x\}.
$$
In both cases, we have
$$
x\in H\subset A
$$
and
$$
H\sim K.
$$
Hence, by Theorem 1.14(iii),
$$
D_{\gamma}(H)=D_{\gamma}(K).
$$
Now define
$$
U:=H\cap D_{\gamma}(K).
$$
Then
$$
x\in U\subset H\subset A.
$$
Moreover, since $D_{\gamma}$ is a lower-density operator, it is idempotent, so
$$
D_{\gamma}(D_{\gamma}(K))=D_{\gamma}(K).
$$
Using Theorem 1.14(iv), we obtain
$$
D_{\gamma}(U)
=
D_{\gamma}(H\cap D_{\gamma}(K))
=
D_{\gamma}(H)\cap D_{\gamma}(D_{\gamma}(K))
=
D_{\gamma}(K)\cap D_{\gamma}(K)
=
D_{\gamma}(K).
$$
Therefore
$$
U\subset D_{\gamma}(K)=D_{\gamma}(U),
$$
which shows that $U\in\tau_{\gamma}$.

Thus $U$ is a $\tau_{\gamma}$-open set such that
$$
x\in U\subset A.
$$
Hence
$$
x\in \operatorname{Int}_{\tau_{\gamma}}(A).
$$
So
$$
A\cap D_{\gamma}(K)\subset \operatorname{Int}_{\tau_{\gamma}}(A).
$$

Combining both inclusions, we conclude that
$$
\operatorname{Int}_{\tau_{\gamma}}(A)=A\cap D_{\gamma}(K).
$$
\end{proof}

\begin{definition}
    A subset $A$ of $\mathbb{R}$ is $\tau_{\gamma}$-discrete if and only if $\forall x \in A$ there exists $\tau_\gamma$-open neighborhood $U_x$ of $x$ such that $U_x\cap A=\{x\}$.
\end{definition}
\begin{theorem}\label{10}
Let $\gamma \in \mathcal{M}$ satisfy Condition $(A)$. Then, for every $A \in L$, the following statements hold:
\begin{enumerate}
\item $\operatorname{Int}_{\tau_\gamma}(A) \sim A$.
\item $A$ is $\tau_\gamma$-regular open if and only if
$$
A=\mathcal{D}_\gamma(A).
$$
\item Let
$$
I_1=\{A\subset \mathbb{R}: |A|=0\},
$$
$$
I_2=\{A\subset \mathbb{R}: A \text{ is } \tau_\gamma\text{-nowhere dense}\},
$$
$$
I_3=\{A\subset \mathbb{R}: A \text{ is of first category in } \tau_\gamma\},
$$
and
$$
I_4=\{A\subset \mathbb{R}: A \text{ is } \tau_\gamma\text{-closed and } \tau_\gamma\text{-discrete}\}.
$$
Then
$$
I_1=I_2=I_3=I_4.
$$
\end{enumerate}
\end{theorem}

\begin{proof}
Since $\gamma$ satisfies Condition $(A)$, Corollary \ref{cor1.15} implies that the operator $\mathcal{D}_\gamma$ is a lower-density operator on $L$. Therefore the topology $\tau_\gamma$ is an abstract density topology generated by a lower-density operator. Hence the assertions follow from the general results on abstract density topologies generated by lower-density operators; see \cite{hwro}.  More precisely, applying those results to the operator $\mathcal{D}_\gamma$, we obtain:
$$
\operatorname{Int}_{\tau_\gamma}(A) \sim A
$$
for every $A \in L$,
$$
A \text{ is } \tau_\gamma\text{-regular open } \Longleftrightarrow A=\mathcal{D}_\gamma(A),
$$
and
$$
I_1=I_2=I_3=I_4.
$$
This completes the proof.
\end{proof}

\begin{corollary}\label{pp5}
   For any $A\in \mathcal{L}$ and $\gamma \in \mathcal{M}$ that satisfy Condition (A), we have $$cl_{\tau_\gamma}(A) \sim A$$
\end{corollary}
\begin{proof}
    The proof is clear from Theorem \ref{10}.
\end{proof}

\begin{theorem}
    The modulus density topology is not regular space.
\end{theorem}
\begin{proof}
Assume that modulus density topology is regular space. Take the $\gamma$-closed set $\mathbb{Q}-\{0\}$ and the point $0\in \mathbb{R}$. Let $U$ and $V$ be $\gamma$-open sets such that $\mathbb{Q}-\{0\}\subset U$ and $0 \in V$. We need to prove that $U\cap V=\emptyset$. Since $V$ is $\gamma$-open, $0\in \mathcal{D}_\gamma (V)$ consequently, $0\in \Phi(V)$ and $$\lim_{\alpha \to 0} \frac{|(-h,h)\cap V|}{|2h|}=1.$$ So for every $\varepsilon>0 $ there exists $\delta >0$ such that for all $0<h<\delta$, $$\frac{|(-h,h)\cap V|}{|2h|}>1-\varepsilon.$$ By taking $\varepsilon= \frac{1}{2}$, we have $$\frac{|(-h,h)\cap V|}{|2h|}>\frac{1}{2}.$$ 

Hence, $|(-h,h)\cap V| >|h|.$ Because of $\mathbb{Q}$ is dense in $\mathbb{R}$, there exists $q\in \mathbb{Q}$, with $|q|< \frac{h}{2}$. i.e. $q\in |(-h,h)\cap V|$ therefore $q\in V$ and also $q\in \mathbb{Q}-\{0\}$ so, $q\in U$ and this shows that $U \cap V\neq \emptyset.$
\end{proof}
\begin{corollary}
    The modulus density topology is not metrizable.
\end{corollary}
\begin{proof}
    Since the modulus density topology is not regular space so it is not metrizable space.
\end{proof}
\section{Modulus approximately continuous function}
In this section we introduce the concept of modulus-approximately continuous functions.
\begin{definition}
    Let $\gamma \in \mathcal{M}$ be a modulus function, and let $f : \mathbb{R} \to \mathbb{R}$.

We say that $f$ is $\gamma$-approximately continuous at a point $x_0 \in \mathbb{R}$ if there exists a Lebesgue-measurable set $A_{x_0} \subset \mathbb{R}$ such that $x_0 \in \mathcal{D}_\gamma(A_{x_0})$ and $
f(x_0)=\lim_{\substack{x\to x_0\\ x\in A_{x_0}}} f(x). $

Then $f$ is called $\gamma$-approximately continuous if it is $\gamma$-approximately continuous at every point of $\mathbb{R}$.
\end{definition}
A useful special case: if $\gamma(t)=t$, then $\mathcal{D}_\gamma$ becomes the usual set of Lebesgue density points, so $\gamma$-approximate continuity reduces to the classical approximate continuity.
\begin{theorem}
    Let $\gamma \in \mathcal{M}$ be a modulus function, and let
$$\mathcal{AC}_\gamma := \{f : \mathbb{R} \to \mathbb{R} : f \text{ is } \gamma\text{-approximately continuous at every } x \in \mathbb{R}\}. $$
Then $\mathcal{AC}_\gamma$, equipped with the usual pointwise addition and scalar multiplication, is a vector space over $\mathbb{R}$.
\end{theorem}
\begin{proof}
We must show that $\mathcal{AC}_\gamma$ is closed under addition and scalar multiplication, and that it contains the zero function.

Let $f,g\in \mathcal{AC}_\gamma$, and fix $x_0\in \mathbb{R}$. Since $f$ is $\gamma$-approximately continuous at $x_0$, there exists a measurable set $A_{x_0}$ such that

$$
x_0\in \mathcal{D}_\gamma(A_{x_0})
\quad \text{and} \quad
\lim_{\substack{x\to x_0\\ x\in A_{x_0}}} f(x)=f(x_0).
$$
Similarly, since $g$ is $\gamma$-approximately continuous at $x_0$, there exists a measurable set $B_{x_0}$ such that
$$
x_0\in \mathcal{D}_\gamma(B_{x_0})
\quad \text{and} \quad
\lim_{\substack{x\to x_0\\ x\in B_{x_0}}} g(x)=g(x_0).
$$
Set $ E_{x_0}:=A_{x_0}\cap B_{x_0}. $ Then $E_{x_0}$ is measurable, and by the intersection property of $\mathcal{D}_\gamma$,
$$ x_0\in \mathcal{D}_\gamma(A_{x_0})\cap \mathcal{D}_\gamma(B_{x_0})=\mathcal{D}_\gamma(E_{x_0}). $$
For every $x\in E_{x_0}$, both $x\in A_{x_0}$ and $x\in B_{x_0}$. Hence, as $x\to x_0$ through $E_{x_0}$,
$$f(x)\to f(x_0) \quad \text{and} \quad g(x)\to g(x_0). $$
Therefore,
$$ (f+g)(x)=f(x)+g(x)\to f(x_0)+g(x_0)=(f+g)(x_0) \quad \text{as } x\to x_0,\ x\in E_{x_0}. $$
Since $x_0\in \mathcal{D}_\gamma(E_{x_0})$, this shows that $f+g$ is $\gamma$-approximately continuous at $x_0$. Because $x_0$ was arbitrary, $f+g\in \mathcal{AC}_\gamma$.

Now let $c\in \mathbb{R}$ and $f\in \mathcal{AC}_\gamma$. Fix $x_0\in \mathbb{R}$, and choose a measurable set $A_{x_0}$ such that $$ x_0\in \mathcal{D}_\gamma(A_{x_0})
\quad \text{and} \quad
\lim_{\substack{x\to x_0\\ x\in A_{x_0}}} f(x)=f(x_0).$$
Then, as $x\to x_0$ through $A_{x_0}$,
$(cf)(x)=cf(x)\to cf(x_0)=(cf)(x_0).$

Thus $cf$ is $\gamma$-approximately continuous at $x_0$. Since $x_0$ was arbitrary, $cf\in \mathcal{AC}_\gamma$. So $\mathcal{AC}_\gamma$ is closed under scalar multiplication.

Finally, consider the zero function $0:\mathbb{R}\to\mathbb{R}$, defined by $0(x)=0$ for all $x\in \mathbb{R}$. Take $A_{x_0}=\mathbb{R}$. Then $x_0\in \mathcal{D}_\gamma(\mathbb{R})$, and clearly

$$
\lim_{\substack{x\to x_0\\ x\in \mathbb{R}}} 0(x)=0=0(x_0).
$$
Hence the zero function belongs to $\mathcal{AC}_\gamma$.

Therefore $\mathcal{AC}_\gamma$ contains the zero vector and is closed under addition and scalar multiplication. Since all remaining vector-space axioms are inherited from the space of all real-valued functions on $\mathbb{R}$ with pointwise operations, $\mathcal{AC}_\gamma$ is a vector space over $\mathbb{R}$.
\end{proof}

\begin{definition}
Let $X\subset \mathbb{R}$ and let $\gamma\in \mathcal{M}$ be a modulus function. A function $ f:X\to \mathbb{R}$

is called bounded $\gamma$-approximately continuous on $X$ if

1. $f$ is $\gamma$-approximately continuous at every point $x_0\in X$, meaning that for each $x_0\in X$ there exists a measurable set $A_{x_0}\subset X$ such that
$$ x_0\in \mathcal{D}_\gamma(A_{x_0}) \quad \text{and} \quad
f(x_0)=\lim_{\substack{x\to x_0,\ x\in A_{x_0}}} f(x), $$
where $\mathcal{D}_\gamma(A_{x_0})$ denotes the set of $\gamma$-density points of $A_{x_0}$,

2. $f$ is bounded on $X$, that is, there exists $M>0$ such that
$$ |f(x)|\leq M \qquad \text{for all } x\in X. $$
Equivalently, one may write

$$ \mathcal{BAC}_\gamma(X)=\{f:X\to \mathbb{R}: f \text{ is bounded and } \gamma\text{-approximately continuous on } X\}.$$

If $X=[a,b]$, then this is the class you use when working with the sup norm
$$ \|f\|_\infty=\sup_{x\in [a,b]} |f(x)|. $$

\end{definition}
\begin{theorem}
Let $X\subset \mathbb{R}$, and let
$$ \mathcal{BAC}_\gamma(X)=\{f:X\to \mathbb{R}: f \text{ is bounded and } \gamma\text{-approximately continuous on } X\}.$$

Then $(\mathcal{BAC}_\gamma(X), \|\cdot\|_\infty)$ is a Banach space, where
$$ \|f\|_\infty=\sup_{x\in X} |f(x)|. $$
\end{theorem}
\begin{proof}
From the previous result, the class of all $\gamma$-approximately continuous functions is a vector space. Since boundedness is preserved under addition and scalar multiplication, $\mathcal{BAC}_\gamma(X)$ is also a vector space.
Now let $(f_n)$ be a Cauchy sequence in $(\mathcal{BAC}_\gamma(X), \|\cdot\|_\infty)$. Since the space $B(X)$ of all bounded real-valued functions on $X$, equipped with the sup norm, is Banach, there exists a bounded function $f:X\to \mathbb{R}$ such that
$$\|f_n-f\|_\infty \to 0.$$
It remains to show that $f$ is $\gamma$-approximately continuous at every point $x_0\in X$.
Fix $x_0\in X$ and let $\varepsilon>0$. Choose $n$ so large that $$ \|f_n-f\|_\infty < \frac{\varepsilon}{3}. $$
Because $f_n$ is $\gamma$-approximately continuous at $x_0$, there exists a measurable set $A_{x_0}\subset X$ such that
$$x_0\in \mathcal{D}_\gamma(A_{x_0})$$
and
$$f_n(x_0)=\lim_{\substack{x\to x_0,\ x\in A_{x_0}}} f_n(x).$$
Hence there exists $\delta>0$ such that whenever $x\in A_{x_0}$ and $|x-x_0|<\delta$, we have
$$ |f_n(x)-f_n(x_0)|<\frac{\varepsilon}{3}.$$
For such $x$,
$$ |f(x)-f(x_0)| \leq |f(x)-f_n(x)|+|f_n(x)-f_n(x_0)|+|f_n(x_0)-f(x_0)| <
\frac{\varepsilon}{3}+\frac{\varepsilon}{3}+\frac{\varepsilon}{3}=\varepsilon.$$
Therefore
$$f(x_0)=\lim_{\substack{x\to x_0,\ x\in A_{x_0}}} f(x).$$
Since also $x_0\in \mathcal{D}_\gamma(A_{x_0})$, this shows that $f$ is $\gamma$-approximately continuous at $x_0$.
Because $x_0$ was arbitrary, $f\in \mathcal{BAC}_\gamma(X)$. Thus every Cauchy sequence in $\mathcal{BAC}_\gamma(X)$ converges in the sup norm to an element of $\mathcal{BAC}_\gamma(X)$. So $(\mathcal{BAC}_\gamma(X), \|\cdot\|_\infty)$ is a Banach space. 

What fails for the full space is boundedness. In fact, even on $X=[0,1]$, one can build an unbounded approximately continuous function by putting very tall continuous triangular bumps on pairwise disjoint tiny intervals accumulating at $0$, with total relative length tending to $0$ near $0$. Then the function is approximately continuous at $0$ along the complement of those intervals, but it is unbounded on $[0,1]$. If $\gamma$ satisfies Condition (A), then $\mathcal{D}_\gamma(A)=\Phi(A)$, so the same example is also $\gamma$-approximately continuous.
\end{proof}
\begin{proposition}
Assume that $\gamma\in \mathcal{M}$ satisfies Condition (A). Let $$ \mathcal{AC}_\gamma([0,1])=\{f:[0,1]\to \mathbb{R}: f \text{ is } \gamma\text{-approximately continuous on } [0,1]\}.
$$
Then $\mathcal{AC}_\gamma([0,1])$ is not a Banach space under the sup norm. In fact, it is not even a normed space under $\|\cdot\|_\infty$.
\end{proposition}
\begin{proof}
For each $n\in \mathbb{N}$, define
$$ a_n=2^{-(n+1)}, \qquad \ell_n=2^{-(2n+2)}, \qquad I_n=(a_n,a_n+\ell_n).$$
These intervals are pairwise disjoint and accumulate only at $0$.
Let $c_n=a_n+\ell_n/2$, and define a triangular bump $\varphi_n:[0,1]\to \mathbb{R}$ by  $$ \varphi_n(x)=
\begin{cases}
\frac{2n}{\ell_n}(x-a_n), & a_n\leq x\leq c_n,\\[4pt]
\frac{2n}{\ell_n}(a_n+\ell_n-x), & c_n\leq x\leq a_n+\ell_n,\\[4pt]
0, & x\notin I_n.
\end{cases}$$
Now define
$$f(x)=\sum_{n=1}^\infty \varphi_n(x), \qquad x\in [0,1]. $$
Since the supports $I_n$ are pairwise disjoint, at each point $x$ at most one summand is nonzero, so $f$ is well defined.
Also, $f(c_n)=n \qquad (n\in \mathbb{N}),$
hence $f$ is unbounded on $[0,1]$. Therefore
$$ \|f\|_\infty=\sup_{x\in [0,1]} |f(x)|=\infty.$$

It remains to show that $f$ is $\gamma$-approximately continuous on $[0,1]$.

For every $x\in (0,1]$, the function $f$ is actually continuous at $x$: away from $0$, only finitely many intervals $I_n$ lie nearby, and each bump $\varphi_n$ is continuous, with value $0$ at the endpoints of $I_n$. Thus $f$ is continuous at every $x\neq 0$, hence $\gamma$-approximately continuous there.
Now consider $x=0$. Put
$$A=[0,1]\setminus \bigcup_{n=1}^\infty I_n. $$
Then $f|_A\equiv 0$, so
$$\lim_{\substack{x\to 0\\ x\in A}} f(x)=0=f(0). $$
Thus it is enough to show that $0$ is a $\gamma$-density point of $A$.
First we show that $0$ is a usual density point of $A$. Let $h\in (2^{-(k+1)}, 2^{-k}]$. Then
$$ (0,h)\cap A^c \subset \bigcup_{n\geq k} I_n, $$
so $$
\lambda\bigl((0,h)\cap A^c\bigr)\leq \sum_{n\geq k}\ell_n
=\sum_{n\geq k}2^{-(2n+2)}
\leq C\,4^{-k} $$

for some constant $C>0$. Since $h>2^{-(k+1)}$, we have $4^{-k}\leq C'h^2$, hence
$$ \lambda\bigl((0,h)\cap A^c\bigr)\leq C''h^2. $$
Therefore
$$\frac{\lambda\bigl((0,h)\cap A^c\bigr)}{h}\to 0
\qquad (h\to 0^+), $$
which implies
$$ \frac{\lambda\bigl((-h,h)\cap A^c\bigr)}{2h}\to 0. $$
So $0$ is a Lebesgue density point of $A$.
Because $\gamma$ satisfies Condition (A), Theorem 1.10 gives

$$\mathcal{D}_\gamma(A)=\Phi(A),$$
where $\Phi(A)$ denotes the set of ordinary density points of $A$. Hence $0\in \mathcal{D}_\gamma(A)$.

Thus $f$ is $\gamma$-approximately continuous at $0$, and consequently $f\in \mathcal{AC}_\gamma([0,1])$.
We have therefore found a $\gamma$-approximately continuous function on $[0,1]$ whose sup norm is infinite. So $\|\cdot\|_\infty$ is not finite on all of $\mathcal{AC}_\gamma([0,1])$. Hence $\mathcal{AC}_\gamma([0,1])$ is not a normed space under the sup norm, and therefore it cannot be a Banach space. 
\end{proof}
\begin{proposition}
Let $\gamma \in M$, let $f:\mathbb{R}\to\mathbb{R}$ be $\gamma$-approximately continuous at a point
$x_{0}\in\mathbb{R}$, and let $\varphi:\mathbb{R}\to\mathbb{R}$ be continuous at $f(x_{0})$.
Then the composition $\varphi\circ f$ is $\gamma$-approximately continuous at $x_{0}$.

Consequently, if $f\in AC_{\gamma}$ and $\varphi:\mathbb{R}\to\mathbb{R}$ is continuous, then
$\varphi\circ f\in AC_{\gamma}$.
\end{proposition}

\begin{proof}
Since $f$ is $\gamma$-approximately continuous at $x_{0}$, there exists a Lebesgue-measurable set
$A_{x_{0}}\subseteq \mathbb{R}$ such that
$$
x_{0}\in D_{\gamma}(A_{x_{0}})
\quad\text{and}\quad
\lim_{\substack{x\to x_{0}\\ x\in A_{x_{0}}}} f(x)=f(x_{0}).
$$
Because $\varphi$ is continuous at $f(x_{0})$, we obtain
$$
\lim_{\substack{x\to x_{0}\\ x\in A_{x_{0}}}} \varphi(f(x))
=
\varphi\!\left(\lim_{\substack{x\to x_{0}\\ x\in A_{x_{0}}}} f(x)\right)
=
\varphi(f(x_{0})).
$$
Hence
$$
\lim_{\substack{x\to x_{0}\\ x\in A_{x_{0}}}} (\varphi\circ f)(x)
=
(\varphi\circ f)(x_{0}).
$$
Since $x_{0}\in D_{\gamma}(A_{x_{0}})$, it follows from Definition 3.1 that $\varphi\circ f$ is
$\gamma$-approximately continuous at $x_{0}$.

For the consequence, assume that $f\in AC_{\gamma}$ and that $\varphi:\mathbb{R}\to\mathbb{R}$ is
continuous. Then $f$ is $\gamma$-approximately continuous at every point $x_{0}\in\mathbb{R}$, and
$\varphi$ is continuous at every point $f(x_{0})$. By the first part, $\varphi\circ f$ is
$\gamma$-approximately continuous at every $x_{0}\in\mathbb{R}$. Therefore,
$\varphi\circ f\in AC_{\gamma}$.
\end{proof}

\section{Conclusion}
\label{Sec:3}

In this paper, we introduced the notion of a $\gamma$-density point for Lebesgue-measurable subsets of $\mathbb{R}$, where $\gamma$ is a modulus function, and developed its basic theory. This notion yields a natural extension of the classical concept of density point. We proved that every $\gamma$-density point is a Lebesgue density point, and that, under Condition~(A), the two notions coincide. As a consequence, for modulus functions satisfying Condition~(A), we obtained a modulus version of the Lebesgue Density Theorem.

We then defined the associated $\gamma$-density topology $\tau_\gamma$ and studied its main properties. In general, $\tau_\gamma$ is contained in the classical Lebesgue density topology $\tau_d$, while under Condition~(A) one has $\tau_\gamma=\tau_d$. We also compared $\tau_\gamma$ with $\psi$-density topologies and showed, under suitable assumptions, that $\tau_\gamma$ is finer than the corresponding $\psi$-density topology. In addition, we established several structural properties of $\tau_\gamma$, including that every countable subset of $\mathbb{R}$ is $\tau_\gamma$-closed, that $(\mathbb{R},\tau_\gamma)$ is nonseparable, that every compact subset is finite, and that the space is neither regular nor metrizable.

Finally, we introduced $\gamma$-approximately continuous functions and showed that they extend the classical notion of approximate continuity. We proved that the class of all $\gamma$-approximately continuous functions forms a vector space, and that the bounded subclass is a Banach space with respect to the supremum norm. We also showed that, even under Condition~(A), the full space of $\gamma$-approximately continuous functions on $[0,1]$ is not a normed space under $\|\cdot\|_\infty$. These results indicate that modulus density provides a flexible framework for extending classical density and continuity notions, and suggests further directions for studying density-type topologies and related function spaces.

\section*{Acknowledgment}
The authors would like to express their sincere gratitude to Prof. Dr. Jacek Hejduk for his insightful comments and valuable suggestions that have significantly improved this manuscript. The corresponding author further acknowledges the generous support of the Türkiye Scholarships Program (YTB), which funded her doctoral studies and thereby contributed substantially to the completion of this research.


\begin{thebibliography}{99}
\bibitem{AP}
Aizpuru, A., Listán-García, M.C., Rambla-Barreno, F.: Density by moduli and statistical convergence. \textit{Quaestiones Mathematicae} \textbf{37}(4), 525--530 (2014). \url{https://doi.org/10.2989/16073606.2014.981683}

\bibitem{fmmto}
Filipczak M., Terepeta M., On continuity concerned with $\psi$-density topologies. Tatra Mt. Math. Publ. 34(2), 2006: 29--36.

\bibitem{gawc}
Goździewicz-Smejda A., Łazarow E., Comparison of $\psi$-Sparse Topologies. Scientific Issues, Mathematics XIV, Jan Dlugosz University in Czestochowa, Czestochowa 2009: 21--36.

\bibitem{cgcntnd}
Goffman C., Neugebauer C. J. and Nishiura T., Density topology and approximate continuity, Duke Math. J. 28(4) 1961, 497--505. \url{https://doi.org/10.1215/S0012-7094-61-02847-2}

\bibitem{hp1}
O. Haupt, C./Ch. Pauc, La topologie approximative de Denjoy envisagée comme vraie topologie, C. R. Acad. Sci. Paris, 234 (1952), 390--392.

\bibitem{hp2}
Haupt O. and Pauc C. Y., Über die durch allgemeine Ableitungsbasen bestimmten Topologien. Ann. Mat. Pura Appl. Ser. IV 36, 1954, 247--271. \url{https://doi.org/10.1007/BF02412841}

\bibitem{hjo}
Hejduk J., On the Density Topologies Generated by Functions. Tatra. Mt. Math. Publ 40(2), 2008, 133--141.

\bibitem{hma}
Hejduk J. and K\"u\c{c}\"ukaslan M., Loranty A., On $<s>$-generalized topologies. Georgian Mathematical Journal, 2024, 31 (3), 437--443. \url{https://doi.org/10.1515/gmj-2023-2096}

\bibitem{ha}
Hejduk J. and Loranty A., On Some Generalized Topologies Satisfying all Separation Axioms. Results in Mathematics, (2024), 79:38. \url{https://doi.org/10.1007/s00025-023-02073-4}

\bibitem{hp}
Hejduk J. and  Nowakowski P., On strong generalized topology generated by the porosity.
Topology and its Applications, 362, (2025). \url{https://doi.org/10.1016/j.topol.2025.109223}

\bibitem{hwro}
Hejduk, J., and Wiertelak, R. On the abstract density topologies generated by lower and almost lower density operators, Traditional and present-day topics in real analysis. Dedicated to Professor Jan Stanisław Lipiński, Filipczak M., Wagner-Bojakowska E. (eds.), Łódź University Press, Łódź, (2013), pp. 431--447. \url{https://doi.org/10.18778/7525-971-1.25}

\bibitem{teavpsi} 
Lazarow E. and Vizváry A., $\psi_{\mathcal{I}}$-Density Topology. Scientific Issues, Mathematics XV, Jan Dlugosz University in Czestochowa, Czestochowa, 2010: 67--80.

\bibitem{LFMC}
León-Saavedra, F., Listán-García, M. del C., Pérez Fernández, F. J., and Romero de la Rosa, M. P. On statistical convergence and strong Cesàro convergence by moduli, Journal of Inequalities and Applications, 2019(1), 298. \url{https://doi.org/10.1186/s13660-019-2252-y}

\bibitem{jcom} 
Oxtoby J. C., Measure and category, Springer Verlag, New YorkHeidelberg-Berlin, 1971. \url{https://doi.org/10.1007/978-1-4615-9964-7}

\bibitem{sjto} 
Taylor S. J., On strengthening the Lebesgue Density Theorem, Fund.Math. 46, 1959, 305--315. \url{https://doi.org/10.4064/fm-46-3-305-315}

\bibitem{tma} 
Terepeta M., A New Approach to $\psi$-Continuity. Tatra Mt. Math. Publ 42(1), 2009, 107--117. \url{https://doi.org/10.2478/v10127-009-0011-z}

\bibitem{mepsi} 
Terepeta M. and Wagner-Bojakowska E., $\psi$-density topology, Circ. Mat. Palermo, Serie $II$, Tomo $XLVIII$, 1999, 451--476. \url{https://doi.org/10.1007/BF02844336}

\bibitem{wewc} 
Wagner-Bojakowska E. and Wilczyński W., Comparison of $\psi$-Density Topologies. Real Analysis Exchange, 25(2), 1999/2000, 661--672.

\bibitem{wewt} 
Wagner-Bojakowska E., and Wilczyński W., The interior operation in $\psi$-density topology. Rendiconti del Circolo Matematico di Palermo, Series II, 49, no. 1, 2000: 5--26. \url{https://doi.org/10.1007/BF02904217}

\bibitem{wh}
Wilczynski W., Density topologies, Chapter 15 in Handbook of Measure Theory, Institute of Mathematics, University of Novi Sad, Elsevier, 2002, 675--702. \url{https://doi.org/10.1016/B978-044450263-6/50016-6}


\end{thebibliography}
\end{document}